\documentclass[11pt,reqno]{amsart}
\usepackage{graphicx}
\usepackage{verbatim}
\usepackage{textcomp}
\usepackage{amssymb}
\usepackage{cite}
\usepackage{amsmath}
\usepackage{latexsym}
\usepackage{amscd}
\usepackage{amsthm}
\usepackage{mathrsfs}
\usepackage{xypic}
\usepackage{bm}
\usepackage{url}
\usepackage{hyperref}

\newtheorem{thm}{Theorem}

\newtheorem{lem}[thm]{Lemma}
\newtheorem{prop}[thm]{Proposition}

\theoremstyle{definition}
\newtheorem{defn}{Definition}[section]
\newtheorem*{ack}{Acknowledgment}
\newtheorem{rem}[thm]{Remark}
\def\R{\mathbb R}

\def\pt{\partial}

\begin{document}
\title[Manifolds with partially horoconvex boundary]{Homotopy type of manifolds with\\ partially horoconvex boundary}

\author{Changwei Xiong}
\address{Mathematical Sciences Institute, Australian National University, Canberra ACT 2601, Australia}
\email{changwei.xiong@anu.edu.au}
\urladdr{maths.anu.edu.au/people/changwei-xiong}

\date{\today}
\thanks{}

\maketitle

\begin{abstract}
Let $M$ be an $n$-dimensional compact connected manifold with boundary, $\kappa>0$ a constant and $1\leq q\leq n-1$ an integer. We prove that $M$ supports a Riemannian metric with the interior $q$-curvature $K_q\geq -q\kappa^2$ and the boundary $q$-curvature $\Lambda_q\geq q\kappa$, if and only if $M$ has the homotopy type of a CW complex with a finite number of cells with dimension $\leq (q-1)$. Moreover, any Riemannian manifold $M$ with sectional curvature $K\geq -\kappa^2$ and boundary principal curvature $\Lambda\geq \kappa$ is diffeomorphic to the standard closed $n$-ball.
\end{abstract}

\section{Introduction}

It is a well-known fact that Riemannian metrics on a manifold will put some restriction on its topology. For example, by the work of Cheeger-Gromoll-Meyer \cite{CG72,GM69}, if a compact Riemannian manifold $M$ has positive sectional curvature and convex nonempty boundary $\pt M$, then it must be a disk. A.~Kasue \cite{Kas83} and R.~Ichida \cite{Ich81}, independently, proved that if $M$ has nonnegative Ricci curvature and mean convex boundary $\pt M$, then either $\pt M$ is connected, or $M$ is isometric to a Riemannian product manifold $\Gamma\times [0,a]$. Later Hung-Hsi~Wu \cite{Wu87} and Ji-Ping~Sha \cite{Sha86} independently generalized Cheeger-Gromoll-Meyer's result to Riemannian manifolds with partially positive curvature. More recently, Huisken and Sinestrari \cite{HS09} were able to use mean curvature flow with surgeries to study the topology of a Euclidean domain with two-convex boundary. See e.g. \cite{BH17,Fra02,MN96,MS85,AC14,JLY14,She93} for other relevant works.

In this short note by following Wu's and Sha's approach we first prove:
\begin{thm}\label{thm1}
An $n$-dimensional compact Riemannian manifold $M$ with non-empty boundary $\pt M$ is diffeomorphic to the standard closed $n$-ball if for some real number $\kappa>0$, the sectional curvature $K$ of $M$ satisfies $K\geq -\kappa^2$ and the minimum eigenvalue $\Lambda$ of the second fundamental form of $\pt M$ satisfies $\Lambda\geq \kappa$.
\end{thm}
\begin{rem}
For the case $\kappa=0$, the same conclusion holds under stronger conditions, e.g. $K\geq 0$, $\Lambda\geq 0$ and $K>0$ in a neighborhood of $\pt M$ \cite{Wu87}. Moreover, it is also proved in \cite{Wu87} that the conclusion holds if $\Lambda>0$ and $\min_{\pt M}\Lambda+\rho_0 \min_MK>0$, where $\rho_0$ is the maximum distance from the boundary. However, this latter condition does not cover the geodesic balls in hyperbolic space with large radius. In this sense, our result is optimal. The same remark applies to the following Theorem \ref{thm2}.
\end{rem}
\begin{rem}
Compared to the standard hyperbolic space, the manifold in Theorem \ref{thm1} may be said to have horoconvex boundary.
\end{rem}

To state our second result, recall the following definitions from \cite{Wu87,Sha86}. Let $M$ be an $n$-dimensional Riemannian manifold with boundary.
\begin{defn}
Let $1\leq q\leq n-1$ be an integer. For any $x\in M$ and any $(q+1)$-orthonormal vectors $\{e_0,e_1,\dots,e_q\}$ in $T_xM$, we call $\sum_{i=1}^qK(e_0\wedge e_i)$ an interior $q$-curvature at $x\in M$. Denote by $K_q(x)$ the minimum of all $q$-curvatures at $x\in M$.
\end{defn}
\begin{defn}
Let $1\leq q\leq n-1$ be an integer. For any $x\in \pt M$, let $\lambda_1(x)\leq \dots \leq \lambda_{n-1}(x)$ be the eigenvalues of the Weingarten operator on $\pt M$. We call any sum $\lambda_{i_1}(x)+\dots+\lambda_{i_q}(x)$ with $1\leq i_1<\dots<i_q\leq n-1$ a boundary $q$-curvature at $x\in \pt M$. Denote by $\Lambda_q(x)$ the minimum of all $q$-curvatures at $x\in \pt M$.
\end{defn}

Then our second result is as follows.
\begin{thm}\label{thm2}
Let $M$ be an $n$-dimensional compact Riemannian manifold with nonempty boundary $\pt M$ and let $1\leq q\leq n-1$ be an integer. Suppose for some $\kappa>0$, we have $K_q\geq -q\kappa^2$ and $\Lambda_q\geq q\kappa$. Then
\begin{enumerate}
  \item[(i)] $M$ has the homotopy type of a CW complex with a finite number of cells with dimension $\leq (q-1)$.
  \item[(ii)] $M$ also has the homotopy type of a CW complex obtained from $\pt M$ by attaching a finite number of cells with dimension $\geq (n-q+1)$.
\end{enumerate}
\end{thm}
\begin{rem}
For the special case $q=2$, under a slightly stronger condition that $K_2\geq -2\kappa^2$ and $\Lambda_2>2\kappa$, Brendle and Huisken \cite{BH17} proved that the manifold is indeed diffeomorphic to a $1$-handlebody. Compared to the sophisticated approach in \cite{BH17}, ours is relatively simple. Besides, from \cite{HS09} to \cite{BH17} we believe the manifolds with certain partial curvature are receiving more and more attention.
\end{rem}

The proofs for Theorems \ref{thm1} and \ref{thm2} follow those in \cite{Wu87,Sha86}. We first use the distance function from the boundary to construct a suitable convex or $q$-convex function, then smooth it and apply the standard Morse theory to conclude the desired results. The methods by Wu \cite{Wu87} and  by Sha \cite{Sha86} are very similar. Here we mainly follow Wu's.

On the other hand, following the argument in \cite{LM84}, Sha in \cite{Sha87} was able to prove an interesting converse result of those in \cite{Wu87,Sha86}. So it is a natural question whether the similar converse holds in our present setting. It turns out that the answer is affirmative. More precisely, we obtain:
\begin{thm}\label{thm3}
Let $M$ be an $n$-dimensional compact connected manifold with nonempty boundary and $1\leq q\leq n-1$ an integer. Assume that $M$ is a handlebody with handles of dimension $\leq (q-1)$. Then for any fixed $\varepsilon>0$ and $c>0$, $M$ admits a Riemannian metric with $\varepsilon$-pinched sectional curvature $K$, i.e. $1\geq K\geq 1-\varepsilon$, and with the boundary $q$-curvature $\Lambda_q> c$. In particular, for any $\kappa>0$, $M$ admits a Riemannian metric with $K_q\geq -q\kappa^2$ and $\Lambda_q\geq q\kappa$.
\end{thm}
\begin{rem}
When $c=0$, the result reduces to those in \cite{LM84,Sha87}.
\end{rem}
The proof of Theorem \ref{thm3} is by combination of Gromov's result \cite{Gro69} and a result regarding handle-attaching process which follows closely those in \cite{LM84,Sha87}. We find that the arguments in \cite{LM84,Sha87} can be nicely adapted to the context here.

The organization of the rest is as follows. In Section \ref{sec2} we review some basic definitions and facts on $q$-convex functions. Then we prove Theorems \ref{thm1} and \ref{thm2} in Sections \ref{sec3} and \ref{sec4}, respectively. Lastly, in Section \ref{sec5} we give the proof of Theorem \ref{thm3}. In Sections \ref{sec2}--\ref{sec4} a Riemannian metric is given, while in Section \ref{sec5} the Riemannian metric is to be constructed.

\begin{ack}
The author is grateful to Ben~Andrews for valuable discussions and to Haizhong~Li for consistent help. This work was supported by the ARC Laureate Fellowship FL150100126.
\end{ack}

\section{Preliminaries}\label{sec2}

In this section we review some basic setting in \cite{Wu87}. Let $M$ be an $n$-dimensional Riemannian manifold and $f:M\rightarrow \R$ a continuous function. Take $x\in M$ and $X\in T_xM$. Define the following extended real numbers:
\begin{align}
Cf(x;X)&=\liminf_{r\rightarrow 0}\frac{f(\exp_x(rX))+f(\exp_x(-rX))-2f(x)}{r^2},\\
&\quad Cf(x)=\inf_{X\in T_xM,\:|X|=1}Cf(x;X).
\end{align}
Note that if $f$ is $C^2$, then $Cf(x;X)$ is just the Hessian $D^2f(X,X)$. We call $f$ is strictly convex if there exists a continuous positive function $h$ on $M$ such that $Cf\geq h$.

Next, given an inner product space $V$ and $\varepsilon>0$, a set of $q$ vectors $\{X_1,\dots,X_q\}\subset V$ is called $\varepsilon$-orthonormal if $|\langle X_i,X_j\rangle-\delta_{ij}|<\varepsilon$ for $i,j=1,\dots,q$. A set of vector fields on a Riemannian manifold is $\varepsilon$-orthonormal if it is so at each point. Now we define the important class $\mathcal{C}(q)$.
\begin{defn}
Given an integer $1\leq q\leq n-1$, the class $\mathcal{C}(q)$ consists of all continuous functions $f:M\rightarrow \R$ which are Lipschitz continuous on each compact subset of $M$ and which have the following property: For each $x_0\in M$, there exists a neighborhood $W$ of $x_0$ and positive constants $\varepsilon$ and $\eta$ such that we have
\begin{equation}
\sum_{i=1}^q Cf(x;X_i)\geq \eta
\end{equation}
for all $x\in W$ and $\{X_1,\dots,X_q\}$ an $\varepsilon$-orthonormal set in $T_xM$.
\end{defn}
From the definition it is easy to see that $\mathcal{C}(1)\subset \mathcal{C}(2)\subset \dots \subset \mathcal{C}(n)$, and that $\mathcal{C}(1)$ is the set of all strictly convex functions. It is also known that $\mathcal{C}(n)$ is the set of locally Lipschitzian strictly subharmonic functions on $M$. One of the important properties of $\mathcal{C}(q)$ is the following smoothing theorem.
\begin{thm}[Smoothing Theorem \cite{GW79,Wu87}]\label{thm4}
On a Riemannian manifold $M$, given an $f\in \mathcal{C}(q)$ ($1\leq q\leq dim\,M$) and a positive continuous function $\xi$, there exists a $C^\infty$ function $F\in \mathcal{C}(q)$ such that $|F-f|<\xi$.
\end{thm}

The next lemma shows how to verify that a locally Lipschitzian function belongs to $\mathcal{C}(q)$. For the statement, given a function $f$ on $M$, a function $h$ is said to support $f$ at $x\in M$ if $h$ is defined in a neighborhood $W$ of $x$ such that $h\leq f$ in $W$ and $h(x)=f(x)$.
\begin{lem}[\cite{Wu87}]\label{lem1}
Given a locally Lipschitzian function $f$ on a Riemannian manifold, $f$ belongs to $\mathcal{C}(q)$ if for any $x_0\in M$, there exists a neighborhood $W$ of $x_0$ such that for each $x\in W$ there exists a $C^\infty$ function $h_x$ supporting $f$ at $x$ which has the following two properties:
\begin{enumerate}
  \item[(i)] For some $\eta>0$, $\sum_{i=1}^q D^2h_x(X_i,X_i)\geq \eta$ holds for every $x\in W$ and every orthonormal set $\{X_1,\dots,X_q\}$ in $T_xM$.
  \item[(ii)] For some constants $A_1,A_2>0$, the inequalities $A_1\leq D^2h_x(X,X)\leq A_2$ hold for every $x\in W$ and every unit tangent vector $X\in T_xM$.
\end{enumerate}
\end{lem}
As remarked after Proposition 2 in \cite{Wu84}, the condition (ii) is easily satisfied. So in the following proofs we will mainly focus on the verification of Condition (i).

\section{Proof of Theorem \ref{thm1}}\label{sec3}

\begin{proof}[Proof of Theorem \ref{thm1}]

Let $\rho:M\rightarrow [0,\infty)$ be the distance to the boundary $\pt M$. We shall prove that $-\log \rho$ is a strictly convex function on $M\setminus \pt M$. That is, given a compact subset $B$ of $M\setminus \pt M$, we will prove that there exists a positive constant $\varepsilon$ such that
\begin{equation}
C(-\log \rho)\geq \varepsilon\quad \text{ on } B.
\end{equation}
For that purpose, let $x\in B$ and let $y\in \pt M$ such that $\rho(x)=d(x,y)=d(x,\pt M)$. Assume $\gamma:[0,b]\rightarrow M$ is a minimizing unit-speed geodesic from $x$ to $y$. So $\gamma(0)=x$, $\gamma(b)=y$ and $\rho(x)=b$. For a unit vector $X\in T_xM$, we shall prove
\begin{equation}
C(-\log \rho)(x;X)\geq \varepsilon,
\end{equation}
where $\varepsilon$ is independent of $x\in B$ and the unit vector $X\in T_xM$.

Let $X(t)(0\leq t\leq b)$ be the parallel translate of $X$ along $\gamma$ with $X(0)=X$. We have the decomposition:
\begin{equation}
X(t)=\alpha X^\perp (t)+\beta \gamma'(t),
\end{equation}
where $\alpha$ and $\beta$ are constants satisfying $\alpha^2+\beta^2=1$, and $X^\perp(t)$ is a parallel unit vector field along $\gamma$ orthogonal to $\gamma'(t)$.

Define a vector field
\begin{equation}
W(t)=\alpha \varphi(t) X^\perp(t)+(1-\frac{t}{b})\beta \gamma'(t),
\end{equation}
where $\varphi(t)=\cosh(\kappa t)$. So $\varphi''=\kappa^2\varphi$.

Note that $W(0)=X(0)=X$ and $W(b)\in T_y(\pt M)$. So we can first take a unit-speed geodesic $\tau:(-s_0,s_0)\rightarrow M$ with $\tau(0)=x$ and $\tau'(0)=X$, and then take $\{\gamma_s(t):s\in (-s_0,s_0), t\in [0,b]\}$ as a one-parameter family of curves $\gamma_s:[0,b]\rightarrow M$ such that $\gamma_0(t)=\gamma(t)$, $\gamma_s(0)=\tau(s)$, $\gamma_s(b)\in \pt M$ and $\frac{\pt}{\pt s}\big|_{s=0}\gamma_s(t)=W(t)$. Furthermore, we require that $\gamma_s(t)$ depends on $s$ in a $C^\infty $ manner.

Let $f(s)$ be the length of $\gamma_s$. Therefore $f(s)\geq \rho(\tau(s))$ and
\begin{equation}
C(-\log \rho)(x;X)\geq C(-\log f)(0).
\end{equation}
So it suffices to prove $C(-\log f)(0)\geq \varepsilon$.

Note that $f\in C^\infty((-s_0,s_0))$. So we have
\begin{equation}
C(-\log f)(0)=(-\log f)''(0)=\frac{f'(0)^2}{f(0)^2}-\frac{f''(0)}{f(0)}.
\end{equation}
It is easy to see that $f(0)=b$ and by the first variation formula for arc length,
\begin{equation}
f'(0)=\langle W(t),\gamma'(t)\rangle \big|_{t=0}^b=-\beta.
\end{equation}
Moreover, by the second variation formula for arc length,
\begin{align*}
f''(0)&=\langle D_{W}W,\gamma'\rangle\big|_{t=0}^b+\int_{0}^b (|D_{\gamma'}(W^\perp(t))|^2-R(W^\perp,\gamma',W^\perp,\gamma'))dt\\
&=-\alpha^2 \varphi(b)^2S(X^\perp(b),X^\perp(b))+\int_0^b \alpha^2( (\varphi')^2-\varphi^2 K(X^\perp\wedge \gamma'))dt.
\end{align*}
So using the assumption in Theorem \ref{thm1}, we have
\begin{align*}
f''(0)&\leq -\alpha^2 \varphi(b)^2\Lambda +\alpha^2 \int_0^b ((\varphi')^2+\kappa^2\varphi^2)dt\\
& =\alpha^2(-\varphi(b)^2\Lambda +\varphi(b)\varphi'(b))=\alpha^2\varphi(b)^2(\kappa\tanh(\kappa b)-\Lambda)\\
&\leq \alpha^2\varphi(b)^2(\kappa \tanh(\kappa \rho_0)-\Lambda)\leq  \alpha^2(\kappa \tanh(\kappa \rho_0)-\Lambda),
\end{align*}
where $\rho_0$ denotes the maximum distance of a point from $\pt M$.

As a consequence, we obtain
\begin{align*}
C(-\log f)(0)&\geq \frac{\beta^2}{b^2}+\frac{\alpha^2(\Lambda-\kappa \tanh(\kappa \rho_0))}{b}\\
&\geq \frac{\beta^2}{\rho_0^2}+\frac{\alpha^2(\Lambda-\kappa\tanh(\kappa \rho_0))}{\rho_0}\\
&\geq \min\{\frac{1}{\rho_0^2},\frac{\Lambda-\kappa\tanh(\kappa \rho_0)}{\rho_0}\}=:\varepsilon(M).
\end{align*}

In summary, we have proven that $-\log\rho$ is strictly convex on $M\setminus \pt M$. The remaining of the proof is exactly the same as the proof of Theorem 2 in \cite{Wu87}, which we omit here.

\end{proof}

\section{Proof of Theorem \ref{thm2}}\label{sec4}

\begin{proof}[Proof of Theorem \ref{thm2}]

We follow the proof of Theorem 1 in \cite{Wu87} and so only give the sketch here. We first show that there exists a $C^\infty$ function $\chi:(-\infty,0)\rightarrow \R$ such that $\chi(-\rho)$ is in $\mathcal{C}(q)$ on $M\setminus \pt M$. Then smooth $\chi(-\rho)$ to finish the proof.

To that end, first recall some notations. Let $\rho_0=\max_M\rho$ as before and let $\rho_0>\rho_1>\rho_2>\dots$ be a sequence of positive numbers such that $\rho_i\rightarrow 0$ as $i\rightarrow \infty$. Let $\mathcal{E}_i=\rho^{-1}([\rho_{i+2},\rho_i])$. Then it is easy to see that for each $x\in M$, there exists an open geodesic ball $\mathcal{B}$ around $x$ and an integer $j$ such that $\mathcal{B}\subset \mathcal{E}_j$.

For any $x\in M$, we fix a minimizing unit-speed geodesic $\gamma:[0,b]\rightarrow M$ from $x$ to the boundary $\pt M$. So $\gamma(0)=x$, $\gamma(b)=y\in \pt M$ and $d(x,\pt M)=d(x,y)=b$. Let $X\in T_xM$. We apply the parallel translate along $\gamma(t)$ to $X$ to get $X(t)$ and decompose it as:
\begin{equation}
X(t)=\alpha X^\perp (t)+\beta \gamma'(t),
\end{equation}
where $\alpha$ and $\beta$ are constants (depending on $X$) satisfying $\alpha^2+\beta^2=|X|^2$, and $X^\perp(t)$ is a parallel unit vector field along $\gamma$ orthogonal to $\gamma'(t)$.

Then we define the vector field
\begin{equation}
W(t)=\alpha \varphi(t) X^\perp(t)+(1-\frac{t}{b})\beta \gamma'(t),
\end{equation}
where $\varphi(t)=\cosh(\kappa t)$ as before. Next we can define the $n$-parameter family of curves $\gamma_{X}:[0,b]\rightarrow M$ ($X\in \exp_x^{-1}\mathcal{B}$) such that (1) $\gamma_0=\gamma$; (2) $\gamma_X(0)=\exp_x(X)$ and $\gamma_X(b)\in \pt M$; and (3) $W(t)$ is induced by the one-parameter family of curves $t\mapsto \gamma_{sX}(t)$ ($-s_0\leq s\leq s_0$). Let $f_x(z)$ be the length of the curve $\gamma_{X}$ where $z=\exp_x(X)\in \mathcal{B}$. Then we get $f_x(x)=b$, $\langle Df_x,X\rangle =-\beta$, and
\begin{align*}
D^2f_x(X,X)&=-\alpha^2\: \varphi(b)^2S(X^\perp(b),X^\perp(b))+\alpha^2\int_0^b ( (\varphi')^2-\varphi^2 K(X^\perp\wedge \gamma'))dt.
\end{align*}
Take an orthonormal set $\{X_k\}_{k=1}^q$ in $T_xM$. There are three cases to consider. Keep in mind $x\in \mathcal{B}\subset \mathcal{E}_j$.

\textbf{Case 1}: $\{X_k\}_{k=1}^q\subset \gamma'(0)^\perp$. So $\alpha_{X_k}=1$ and we can easily get from the assumption in Theorem \ref{thm2}
\begin{equation}
\sum_{k=1}^q D^2f_x(X_k,X_k)\leq -\beta_j<0
\end{equation}
for some positive constant $\beta_j$. Since
\begin{equation}
D^2(\chi(-f_x))=\chi''(-f_x)df_x\otimes df_x-\chi'(-f_x)D^2f_x,
\end{equation}
we obtain
\begin{align*}
\sum_{k=1}^q & D^2(\chi(-f_x))(X_k,X_k)\geq \beta_j \min\chi'|_{[-\rho_j,-\rho_{j+2}]}.
\end{align*}
Here and below we assume that $\chi'>0$ and $\chi''>0$ on $[-\rho_0,0)$.

\textbf{Case 2}: We choose $\theta_j>0$ small such that if $|\langle X_k,\gamma'(0)\rangle| \leq \theta_j$ for all $1\leq k\leq q$, then
\begin{align*}
\sum_{k=1}^q & D^2f_x(X_k,X_k)\leq  -\frac{\beta_j}{2}.
\end{align*}
The existence of such $\theta_j$ is proved in Lemma 7 of \cite{Wu87}. Roughly speaking, it follows from the fact that Case 2 can be viewed as a perturbation of Case 1. As a consequence, we have
\begin{align*}
\sum_{k=1}^q & D^2(\chi(-f_x))(X_k,X_k)\geq \frac{\beta_j}{2} \min\chi'|_{[-\rho_j,-\rho_{j+2}]}.
\end{align*}

\textbf{Case 3}: There exists at least one of $\{X_k\}_{k=1}^q$, say $X_1$, satisfying $|\langle X_1,\gamma'(0)\rangle |>\theta_j$, we have
\begin{align*}
D^2(\chi(-f_x))(X_1,X_1)&\geq \chi''(-\rho(x))\theta_j^2-\chi'(-\rho(x))P_j,\\
\sum_{k=2}^q D^2(\chi(-f_x))(X_k,X_k)&\geq -\chi'(-\rho(x))(q-1)P_j,
\end{align*}
where $P_j$ is a positive constant such that $D^2f_x(X,X)\leq P_j$ for any unit $X\in T_xM$ (see Lemma 6 in \cite{Wu87}). Therefore, if we assume further that
\begin{equation}
\chi''(t)\geq \frac{\chi'(t)}{\theta_j^2}(\beta_j+qP_j)
\end{equation}
for any $t\in [-\rho_j,-\rho_{j+2}]$, then
\begin{align*}
\sum_{k=1}^q & D^2(\chi(-f_x))(X_k,X_k)\geq \beta_j \min\chi'|_{[-\rho_j,-\rho_{j+2}]}.
\end{align*}

As a summary of these three cases, there exists a sequence of positive constant $\{\Gamma_0,\Gamma_1,\dots\}$ such that for each $x\in \mathcal{E}_j$ and any orthonormal set $\{X_1,\dots,X_q\}$ in $T_xM$, we have
\begin{align*}
\sum_{k=1}^q & D^2(\chi(-f_x))(X_k,X_k)\geq \Gamma_j.
\end{align*}
Then it follows from Lemma \ref{lem1} that $\chi(-\rho)\in \mathcal{C}(q)$, and we can use the Smoothing Theorem \ref{thm4} to finish the proof. The details, including the construction of the required $\chi$, are the same as in the proof of Theorem 1 in \cite{Wu87} and are omitted here.

\end{proof}

\section{Proof of Theorem \ref{thm3}}\label{sec5}

In this section we first outline the proof of Theorem \ref{thm3} and then prove the crucial Proposition \ref{prop1} used in the proof.

\begin{proof}[Proof of Theorem \ref{thm3}]
Let $M$ be as in Theorem \ref{thm3}. By Gromov's result \cite{Gro69}, $M$ supports a Riemannian metric with $\varepsilon$-pinched sectional curvature. Fix this metric as the ambient metric. We claim that restricted on some suitable neighborhood $N$ of the spine of the handlebody, the metric is such that $\Lambda_q(\pt N)>c$ for any fixed $c>0$. Therefore, since $N$ is diffeomorphic to $M$, we get a Riemannian metric on $M$ via this diffeomorphism which has $\varepsilon$-pinched sectional curvature and $\Lambda_q>c$.

We prove the claim as follows. Start with a $0$-handle $D^0$ in $M$, which is diffeomorphic to a disk. Without loss of generality, we may assume $D^0$ is a small geodesic ball such that on $\pt D^0$ we have $\Lambda_q>c$. Then we attach all other $k$-handles ($k\leq q-1$) to $D^0$ as in Proposition~\ref{prop1} below. Proposition~\ref{prop1} guarantees that the handlebody we get after every handle-attaching still satisfies $\Lambda_q>c$. After all the handle-attachings we get the desired neighborhood $N$ of the spine of the handlebody. So the proof of Theorem~\ref{thm3} would be complete.

\end{proof}

It remains to prove the following.

\begin{prop}\label{prop1}
Let $X$ be an oriented hypersurface in an $n$-dimensional Riemannian manifold $\Omega$ with $\Lambda_q(X)>c$ for some integer $1\leq q\leq n-1$ and constant $c>0$. Suppose $X'$ is a hypersurface obtained from $X$ by attaching a $k$-handle $D^k$ to the positive side (see the definition below) of $X$ with $k\leq q-1$. Then $X'$ can be constructed such that $\Lambda_q(X')>c$.
\end{prop}

The rest of this section is devoted to the proof of Proposition \ref{prop1}. To that end we need some preliminaries.

First note that from the assumption, $X$ is mean convex. We choose a unit normal $\nu$ on $X$ such that the mean curvature given by $H=div_X\nu$ is positive and we call the side $X^+\subset \Omega$ into which $\nu$ points the positive side. Denote the negative side by $X^-\subset \Omega$.

The core of the $k$-handle $D^k$ is a $k$-dimensional disk, which is also denoted by $D^k$ for simplicity. Assume the $k$-disk $D^k$ is attached orthogonally to $X$ from the positive side. Set
\begin{equation}
s(x)=dist(x,X),\quad r(x)=dist(x,D^k), \quad x\in \Omega.
\end{equation}
Then $s$ is smooth in $\Omega_1'=\Omega_1\setminus X^-$ where $\Omega_1$ is some neighborhood of $X$ in $\Omega$, and $|\nabla s|=1$. Similarly, $r$ is smooth in $\Omega_2'=\Omega_2\setminus (X^-\cup D^k)$ where $\Omega_2$ is a suitable neighborhood of $D^k$ in $\Omega$, and $|\nabla r|=1$.

It is easy to see that the map
\begin{equation}
(r,s):\Omega_1'\cap \Omega_2'\rightarrow \R^2
\end{equation}
is a smooth submersion. The idea here is the same as in \cite{LM84,Sha87}, i.e. to construct a regular curve $\gamma$ in $\R^2$, such that the hypersurface $S_\gamma$ defined as the inverse image $(r,s)^{-1}(\gamma)$ joins the hypersurface $r^{-1}(\varepsilon_0)$ to $X$ smoothly for some $\varepsilon_0>0$, and the new hypersurface $X'$ obtained in this way will still satisfy $\Lambda_q(X')>c$.

The joint hypersurface $S_\gamma$ will be chosen as the level set of the smooth function $F(x)=s(x)-f(r(x))=0$ for some suitable smooth decreasing function $f$. Given $f$, the unit normal of $S_\gamma$ is
\begin{equation}
e_n=\frac{\nabla F}{|\nabla F|},\text{ with }\nabla F=\nabla s-f'(r)\nabla r.
\end{equation}
Then the second fundamental form of $S_\gamma$ is given by
\begin{equation}
B_F(v,w)=\langle \nabla_v e_n,w\rangle
\end{equation}
for any $v,w\in TS_\gamma$.

Let $e_1,\dots,e_q$ be any orthonormal vectors in $TS_\gamma$. Then as estimated in \cite[p.~360]{Sha87}, we obtain
\begin{align*}
\sum_{i=1}^q B_F(e_i,e_i)&\geq \frac{1}{|\nabla F|}\left(\sigma_s(q+1)-f'(r)\sigma_r(q+1)-f''(r)\sum_{i=1}^q(\nabla_{e_i}r)^2\right)\\
&+\frac{1}{|\nabla F|^3}\left(f'(r)\nabla^2 r(\nabla s,\nabla s)-f'(r)^2 \nabla^2s(\nabla r,\nabla r)\right),
\end{align*}
where $\sigma_u(m)$ denotes the sum of the least $m$'s eigenvalues of the Hessian $\nabla^2 u$. In other words, denote the eigenvalues of $\nabla^2 u$ by
\begin{equation}
\lambda_1\leq \lambda_2\leq \dots\leq \lambda_n.
\end{equation}
Then $\sigma_u(m)=\lambda_1+\dots+\lambda_m$. Note that for a smooth function $u$ satisfying $|\nabla u|=1$, the level hypersurface of $u$ satisfies $\Lambda_q>c$ if and only if $\sigma_u(q+1)>c$. See \cite{Sha87} for more details.

Now we recall the following lemma, which is Lemma 2 in \cite{Sha87} with minor modification.
\begin{lem}
Let the setting be as in Proposition \ref{prop1}. (i) We can choose $\Omega_1$ such that there exists a constant $\delta>0$ such that $\sigma_s(q+1)\geq c+2\delta$ in $\Omega_1$. (ii) We can choose $\Omega_2$ such that $\sigma_r(q+1)>c_1/r$ in $\Omega_2'$, where $c_1>0$ is a constant.
\end{lem}
We mention that (i) follows from $\Lambda_q(X)>c$ and (ii) is proved by use of a calculation in Fermi coordinates and the assumption $k\leq q-1$.

Now we are ready to finish the proof of Proposition \ref{prop1}. As in \cite{Sha87}, we work in a small neighborhood $U$ around the joint
\begin{equation}
U=\{x\in D_{2\varepsilon_1}\cap X_{2\varepsilon_2}\cap X^+:r(x)>0\},
\end{equation}
where $D_\varepsilon=\{x\in \Omega:r(x)<\varepsilon\}$ and $X_\varepsilon=\{x\in \Omega:s(x)<\varepsilon\}$. Here $\varepsilon_1$ and $\varepsilon_2$ are small positive constants. And by the construction in \cite{Sha87}, for given $\varepsilon_1$, $\varepsilon_2$ and $c_0=c_0(\varepsilon_1,\varepsilon_2)$, it suffices to prove
\begin{equation}
\sum_{i=1}^q B_F(e_i,e_i)>c
\end{equation}
on the part of $S_\gamma$ where $\varepsilon_0\leq r\leq \varepsilon_1$ for some chosen $\varepsilon_0=\varepsilon_0(\varepsilon_1,c_0)>0$.

With the help of the lemma above, as arguing in the last steps in \cite{Sha87}, we conclude that there exist $\varepsilon_1$, $\varepsilon_2$ and $c_0$ such that
\begin{equation}
\sum_{i=1}^q B_F(e_i,e_i)\geq \frac{1}{|\nabla F|}\left(c+\delta-\frac{c_0f'(r)}{r}\right)+\frac{1}{|\nabla F|}\left(\delta-\frac{c_0f'(r)}{r}-f''(r)\right),
\end{equation}
when $a\varepsilon_0\leq r\leq \varepsilon_1$, and
\begin{equation}
\sum_{i=1}^q B_F(e_i,e_i)\geq \frac{1}{|\nabla F|}\left(c+\delta-\frac{c_0f'(r)}{r}\right)+\frac{1}{|\nabla F|}\left(\delta-\frac{c_0f'(r)}{r}-\frac{f''(r)}{f'(r)^2}\right),
\end{equation}
when $\varepsilon_0\leq r\leq a\varepsilon_0$. Here $a>1$ is a constant depending on $\varepsilon_1$ and $c_0$.

Note that $|\nabla F|=\sqrt{1+f'(r)^2-2f'(r)\langle \nabla r,\nabla s\rangle}$. So when $|f'(r)|$ is small enough (say $|f'(r)|\leq \tau$ for some $\tau>0$), we have
\begin{equation}\label{eq10}
\frac{1}{|\nabla F|}\left(c+\delta-\frac{c_0f'(r)}{r}\right)>c.
\end{equation}
On the other hand, when $|f'(r)|\geq \tau$, the dominant term in \eqref{eq10} is $-\frac{1}{|\nabla F|}\frac{c_0f'(r)}{r}$. So we still have
\begin{equation}
\frac{1}{|\nabla F|}\left(c+\delta-\frac{c_0f'(r)}{r}\right)>c,
\end{equation}
provided $\varepsilon_1$ and $\varepsilon_2$ are small so that $1/r$ is large enough.

Meanwhile, as in \cite{Sha87}, the function $f$ can be constructed such that
\begin{equation}
\delta-\frac{c_0f'(r)}{r}-f''(r)>0,\quad a\varepsilon_0\leq r\leq \varepsilon_1,
\end{equation}
and
\begin{equation}
\delta-\frac{c_0f'(r)}{r}-\frac{f''(r)}{f'(r)^2}>0,\quad \varepsilon_0\leq r\leq a\varepsilon_0.
\end{equation}
So Proposition \ref{prop1} follows immediately.


\bibliographystyle{Plain}

\end{document}